\documentclass[10pt, a4paper]{article}
\usepackage[a4paper, total={6in, 9in}]{geometry}
\usepackage{amsmath}
\usepackage{amssymb}
\usepackage{amsthm}
\usepackage[utf8]{inputenc}
\usepackage{tikz,pgfplots}
\usepackage{enumerate}
\usepackage{tabularx}
\usepackage{multirow}
\usepackage{rotating}
\usepackage{graphicx}  
\definecolor{mpigreen}{HTML}{5C871D}    % include this line if your document
\usepackage{caption,subcaption}
\usepackage[vlined,longend,linesnumbered,tworuled]{algorithm2e}
\newcommand{\smb}{\left[\begin{smallmatrix}}
\newcommand{\sme}{\end{smallmatrix}\right]}

\usepackage{ mathdots }
\usepackage[colorlinks=true]{hyperref}

\newcounter{mymac@matlab}
  \newcommand{\matlab}{MATLAB% 
   \ifnum\value{mymac@matlab}<1%
   \textsuperscript{\textregistered}%
   \setcounter{mymac@matlab}{1}%
   \fi%
  }

\newcommand{\diag}[1]{\ensuremath{\mathop{\mathrm{diag}}\left( #1 \right)}}
\newcommand{\intel}{Intel\textsuperscript{\textregistered}}
\newcommand{\xeon}{Xeon\textsuperscript{\textregistered}}
\newcommand{\range}[1]{\ensuremath{\mathop{\mathrm{range}}\left( #1 \right)}}
\newcommand{\Real}[1]{\ensuremath{\mathop{\mathrm{Re}}\,\left(#1\right)}}
\newcommand{\Imag}[1]{\ensuremath{\mathop{\mathrm{Im}}\,\left(#1\right)}}
\newcommand{\half}{\ensuremath{\frac{1}{2}}}
\newcommand{\myspan}[1]{\ensuremath{\mathop{\mathrm{span}}\left\lbrace #1 \right\rbrace}}
\newcommand{\intab}[2]{\ensuremath{\int\limits_{#1}^{#2}}}

\newcommand\R{\ensuremath{\mathbb{R}}}
\newcommand{\Rn}{\ensuremath{\R^{n}}}
\newcommand{\Rp}{\ensuremath{\R^{p}}}
\newcommand{\Rm}{\ensuremath{\R^{m}}}
\newcommand{\Rr}{\ensuremath{\R^{r}}}
\newcommand{\Rnn}{\ensuremath{\R^{n\times n}}}
\newcommand{\Rpn}{\ensuremath{\R^{p\times n}}}
\newcommand{\Rnm}{\ensuremath{\R^{n\times m}}}

\newcommand{\Rrr}{\ensuremath{\R^{r\times r}}}

\newcommand{\Rrm}{\ensuremath{\R^{r\times m}}}

\newcommand{\Rpr}{\ensuremath{\R^{p\times r}}}
\newcommand\C{\ensuremath{\mathbb{C}}}

\newcommand{\Cnn}{\ensuremath{\C^{n\times n}}}

\newcommand{\Cnm}{\ensuremath{\C^{n\times m}}}
\newcommand{\argmax}{\ensuremath{\mathop{\mathrm{argmax}}}}
\newcommand{\tA}{\ensuremath{\tilde{A}}}
\newcommand{\tB}{\ensuremath{\tilde{B}}}
\newcommand{\tC}{\ensuremath{\tilde{C}}}
\newcommand{\tH}{\ensuremath{\tilde{H}}}

\newcommand{\ty}{\ensuremath{\tilde{y}}}
\newcommand{\tx}{\ensuremath{\tilde{x}}}
\newcommand{\cC}{\ensuremath{\mathcal{C}}}

\newcommand{\cE}{\ensuremath{\mathcal{E}}}
\newcommand{\cH}{\ensuremath{\mathcal{H}}}
\newcommand{\cK}{\ensuremath{\mathcal{K}}}

\newcommand{\cR}{\ensuremath{\mathcal{R}}}
\newcommand{\cT}{\ensuremath{\mathcal{T}}}

\newcommand{\hA}{\ensuremath{\hat{A}}}
\newcommand{\hB}{\ensuremath{\hat{B}}}
\newcommand{\hC}{\ensuremath{\hat{C}}}

\newcommand{\hV}{\ensuremath{\hat{V}}}
\newcommand{\hW}{\ensuremath{\hat{W}}}

\newcommand{\bS}{\ensuremath{\bm{S}}}
 \newcommand{\bm}[1]{\mathbf{#1}}

\hypersetup{%
pdftitle = {Balanced truncation model order reduction in limited time for large systems.}, %
pdfsubject = {Matrix Equations, \today}, %
pdfauthor = {Patrick K\"urschner}, %
pdfkeywords = {Lyapunov equation, rational Krylov subspaces, model order reduction, balanced truncation, matrix exponential}
}

\theoremstyle{definition}
\newtheorem{definition}{Definition}[section]
\newtheorem{theorem}{Theorem}[section]
\newtheorem{lemma}[theorem]{Lemma}

\newcommand{\expm}[1]{\ensuremath{\mathop{\mathrm{e}^{#1}}}}
\title{Balanced truncation model order reduction in limited time intervals for large systems} %
\author{Patrick K\"{u}rschner \thanks{Max Planck Institute for Dynamics of Complex Technical
Systems,~Sandtorstra{\ss}e~1,~39106~Magdeburg,~Germany,~
\texttt{kuerschner@mpi-magdeburg.mpg.de}.}}

\begin{document}
% \linenumbers
\maketitle
  \begin{abstract}
  In this article we investigate model order reduction of large-scale systems
  using time-limited balanced truncation, which restricts the well known
  balanced truncation framework to prescribed finite time intervals. The main
  emphasis is on the efficient numerical realization of this model reduction
  approach in case of large system dimensions.  We discuss numerical methods to deal with the resulting
  matrix exponential functions and Lyapunov equations which are
  solved for low-rank approximations.
  Our
main tool for this purpose are rational Krylov subspace methods.
  We also discuss the eigenvalue decay and numerical rank of the solutions of the Lyapunov equations. These results, and also
numerical experiments, will show that depending on the final time horizon, the numerical rank of the Lyapunov solutions in
  time-limited balanced truncation can be smaller compared to standard
  balanced truncation. 
  In numerical experiments we test the approaches for computing low-rank factors of the involved Lyapunov solutions and illustrate
  that time-limited balanced truncation can generate reduced order models having a higher accuracy in the considered time region.
%   \keywords{Lyapunov equation \and rational Krylov subspaces \and model order reduction \and balanced truncation \and matrix exponential}
% \subclass{15A16 \and  15A18 \and  15A24 \and  65F60 \and  93A15 \and  93C}
\end{abstract}

\section{Introduction}
%%%%%%%%%%%%%%%%%%%%%%%%%%%%%%%%%%%%%%%%%%%%%%%%%%%%%%%%%%%%%%%%%%%%%%%%%%%%%%%%
\subsection{Model reduction of linear systems}
Consider continuous-time, linear, time-invariant (LTI) systems 
\begin{subequations}\label{morBT:sys}
  \begin{align}
    \dot x(t)&=Ax(t)+Bu(t),\quad x(0)=0,\\
    y(t)&=Cx(t)
  \end{align}
\end{subequations}
with $A\in\Rnn$, $B\in\Rnm$, and $C\in\Rpn$. Typically, the vector functions $x(t)\in\Rn$, $u(t)\in\Rm$, and $y(t)\in\Rp$ are referred to as state, control,
and, output vector, respectively. We assume that $m,p\ll n$ and $A$ is Hurwitz,
i.e., $\Lambda(A)\subset\C_-$, such that~\eqref{morBT:sys} is asymptotically stable. Given a large state space dimension $n$, model order reduction aims
towards finding a reduced order model
\begin{subequations}\label{morBT:rom}
  \begin{align}
    \dot \tx(t)&=\tA\tx(t)+\tB u(t),\quad \tx(0)=0,\\
    \ty(t)&=\tC\tx(t)
  \end{align}
\end{subequations}
with $A\in\Rrr$, $B\in\Rrm$, $C\in\Rpr$, $\tx(t)\in\Rr$ and a drastically reduced state dimension $r\ll n$. The smaller reduced system~\eqref{morBT:rom}
should approximate the input-output behavior or the original system~\eqref{morBT:sys} well. Moreover, simulating the system, i.e., solving the differential
equations in~\eqref{morBT:rom} for many different control functions $u(t)$ should be numerically much less expensive compared to the original
system~\eqref{morBT:sys}.
For the approximation of~\eqref{morBT:sys} regarding the actual time domain behavior, it is desired that for all feasible input functions $u(t)$,
\begin{subequations}\label{morBT:approx_inf}
\begin{align}\label{morBT:timeapprox}
 \ty(t)\approx y(t)\quad\text{for}\quad t\geq0,
\end{align}
in other words, $\|y(t)-\ty(t)\|$ should be small $\forall t\geq0$ in some norm $\|\cdot\|$. With the help of the Laplace transformation, one can also formulate
the approximation problem in the frequency domain, e.g., via
\begin{align}\label{morBT:freqapprox}
 \tH(\imath\omega)&\approx H(\imath\omega)\quad\text{for}\quad \omega\in\R,\quad \imath^2=-1,\\\nonumber
 \quad\text{where}\quad H(s)&=C(sI-A)^{-1}B,\quad \tH(s)=\tC(sI_r-\tA)^{-1}\tB
\end{align}
are the transfer function matrices of~\eqref{morBT:sys} and~\eqref{morBT:rom}.
\end{subequations}
There exist different model order reduction technologies and here we focus on balanced truncation (BT)~\cite{morMoo81} model order reduction. The backbone of
BT are 
the infinite controllability and observability Gramians of
\eqref{morBT:sys} which are the
the  symmetric, positive semidefinite solutions $P_{\infty},Q_{\infty}$ of the continuous-time, algebraic Lyapunov
equations
\begin{align}\label{morBT:gramians}
  AP_{\infty}+P_{\infty}A^T=-BB^T,\quad A^TQ_{\infty}+Q_{\infty}A=-C^TC.
\end{align}
The Hankel singular values (HSV) of~\eqref{morBT:sys} are the eigenvalues
of the product $P_{\infty}Q_{\infty}$ and are system invariants under state space transformations. 
The magnitude of the HSV enables to identify components that are weakly
controllable and observable. In BT this is achieved by first transforming~\eqref{morBT:sys}
into a balanced realization such that
$P_{\infty}=Q_{\infty}=\Sigma_{\infty}=\diag{\sigma_1,\ldots,\sigma_n}$. Dropping all states corresponding to small $\sigma_j$ gives the reduced order
model. Solving~\eqref{morBT:gramians} for the Gramians $P_{\infty},~Q_{\infty}$ is the computationally most challenging part
of balanced truncation. For large-scale systems one therefore uses low-rank approximations of the Gramians instead, e.g., $Z_PZ_P^T\approx P_{\infty}$ with
low-rank
solution factors $Z_P\in\R^{n \times k_P}$, rank$(Z_P)=k_P\ll n$, and likewise for $Q_{\infty}$. This strategy is backed up by the often numerically observed
and theoretically explained
rapid eigenvalue decay of solutions of Lyapunov equations~\cite{AntSZ02,BakES15,Gra04,Pen00,TruV07}.
% which causes $P_{\infty},Q_{\infty}$ to have a small
% numerical rank. 
The computation of the low-rank factors $Z_P,Z_Q$ of $P_{\infty},Q_{\infty}$ can be done efficiently by state-of-the-art numerical algorithms for solving
large Lyapunov equations~\cite{BenS13,Sim16}. BT using low-rank factors $Z_P,Z_Q$ of the Gramians~\eqref{morBT:gramians} is illustrated in
Algorithm~\ref{alg:lrsrbt}.
%%%%%%%%%%%%%%%%%%%%%%%%%
\begin{algorithm}[t]
  \caption{Square-root balanced truncation with low-rank factors}
  \label{alg:lrsrbt}
  \SetEndCharOfAlgoLine{} \SetKwInOut{Input}{Input}\SetKwInOut{Output}{Output}
  % \begin{algorithmic}[1]
  \Input{System matrices $A$,~$B$,~$C$ defining an asymptotically stable
    dynamical system~\eqref{morBT:sys}.}
  % , truncation tolerance $\varepsilon_{\text{BT}}$}
  \Output{Matrices $\tA$,~$\tB$,~$\tC$ of the reduced system.}  
  Compute $Z_P,~Z_Q$ (e.g., with the methods described in~\cite{BenS13,Sim16}), such  that $Z_PZ_P^T\approx P_{\infty}$, $ Z_QZ_Q^T\approx Q_{\infty}$
in~\eqref{morBT:gramians}.\nllabel{alg:bt_lyap}\;
  \nllabel{alg:BT_firstord_lyap} Compute thin singular value
  decomposition
  \begin{equation*}%\label{BT_svd}
    Z_Q^TZ_P=X\Sigma Y^T =\begin{bmatrix} X_1&X_2 \end{bmatrix} \diag{\Sigma_1,~\Sigma_2}
    \begin{bmatrix} Y_1&Y_2 \end{bmatrix}^T%\vspace{-1em}
  \end{equation*}
  with $\Sigma_1=\diag{\sigma_1,\ldots,\sigma_r}$ containing the largest $r$
  (approximate) HSV.\; \nllabel{alg:BT_firstord_svd}
  Construct $T:=Z_P Y_1\Sigma_1^{-\half}$ and $S:=Z_QX_1\Sigma_1^{-\half}$.\;
  Generate reduced order model
  \begin{equation}\label{alg:btmor_proj}
    \tA:=S^TAT, \quad \tB:=S^TB, \quad \tC:=CT.
  \end{equation}\;\nllabel{alg:BT_rom}
\end{algorithm}
The $\cH_{\infty}$-norm 
\begin{align*}
 \|H\|_{\cH_{\infty}}=\sup_{\omega\in\R}(\|H(\imath\omega)\|_2). 
 \end{align*}
of a stable system~\eqref{morBT:sys} is the $L_2$ induced norm of the convolution operator.
It connects to the time-domain behavior via $\|y\|_{L_2}\leq \|H\|_{\cH_{\infty}}\|u\|_{L_2}$.
With exact Gramian factors, i.e., $Z_PZ_P^T=P_{\infty}$, $Z_QZ_Q^T=Q_{\infty}$, BT is known to always generate a asymptotically stable reduced system for
which the error bound 
\begin{align}\label{morBT:error}
  \|H-\tH\|_{\cH_{\infty}}\leq 2\sum\limits_{j=r+1}^n\sigma_j
\end{align}
holds.  
%%%%%%%%%%%%%%%%%%%%%%%%%%%%%%%%%%%%%%%%%%%%%%%%%%%%%%%%%%%%%%%%%%%%%%%%%%%%
\subsection{Model reduction in limited time- and frequency intervals}
The approximation paradigms~\eqref{morBT:approx_inf} enforce that the reduced system~\eqref{morBT:rom} is accurate for all times $t\in\R_+$ and frequencies
$\omega\in\R$. From a practical point of view, achieving~\eqref{morBT:approx_inf} might overshoot a realistic objective. For instance, if~\eqref{morBT:sys}
models a mechanical or electrical system, practitioners working with this model (and its approximation~\eqref{morBT:rom}) might only be interested in certain
finite frequency intervals $0\leq\omega_1<\omega_2<\infty$. Likewise, when the final goal is to carry out simulations of~\eqref{morBT:sys}, i.e., acquire
time-domain solutions for various controls $u(t)$, one is usually only interested in $y(t)$ for $t$ smaller than some final time $t_e<\infty$. Hence, we
consider time-
and frequency restricted versions of~\eqref{morBT:approx_inf} of the form
\begin{subequations}\label{morBT:approx_lim}
\begin{alignat}{2}\label{morBT:approx_limT}
 \ty(t)&\approx y(t)\quad&\text{for}\quad t&\in\cT\subset\R_+,\\\label{morBT:approx_limF}
 \tH(i\omega)&\approx H(i\omega)\quad&\text{for}\quad \omega&\in\Omega\subset\R,~\Omega=-\Omega,
\end{alignat}
\end{subequations}
where the time- and frequency regions $\cT,\Omega$ of interest should be provided by the underlying application.
The expressions~\eqref{morBT:approx_lim} demand that the reduced order model~\eqref{morBT:rom} is only accurate in $\cT,\Omega$ but allow larger
errors outside these regions. Compared to MOR approaches for the unrestricted setting~\eqref{morBT:approx_inf}, it is desired that MOR
approaches for~\eqref{morBT:approx_lim} achieve smaller approximation errors in $\cT,\Omega$ with the same reduced order $r$. Alternatively, one demands that 
  time- and frequency restricted MOR leads to comparable approximation errors in $\cT,\Omega$ with reduced systems of smaller order $r$.
 A secondary question is how the added time- and frequency restrictions influence the computational effort compared to an unrestricted MOR
method of the same type. This issue will be in our particular focus. 
Typically, the time region in~\eqref{morBT:approx_limT} has the form 
\begin{subequations}\label{morBT:approx_timelim}
\begin{align}\label{morBT:approx_timelimT}
 \cT=[0,t_e],\quad t_e<\infty
\end{align}
which will also be the main situation in this work, but the more general restriction
\begin{align}\label{morBT:approx_timelimT12}
 \cT=[t_s,t_{e}],~0<t_s<t_{e}<\infty
\end{align}
\end{subequations}
will also be briefly discussed. Regarding the frequency restricted setting~\eqref{morBT:approx_limF}, the typical regions
are 
\begin{align*}%\label{morBT:approx_freqlim}
 \Omega:=\hat\Omega\cup-\hat\Omega, \hat\Omega:=\bigcup\limits_{i=1}^h[\omega_i,\omega_{i+1}]\quad\text{with}\quad
0\leq\omega_1<\ldots<\omega_{h}<\omega_{h+1}<\infty.
\end{align*}
Introducing time- or frequency restrictions into balanced truncation MOR has been originally proposed in~\cite{morGawJ90} and further studied in, e.g.,
\cite{morBenKS16,morBenD16,morFehEtal13,morGugA04}.
In certain applications, e.g., circuit design, only single frequencies $\omega\in\R$ might be of interest and an associated version of balanced truncation is
addressed in~\cite{BenDYetal13}. $\cH_2$-MOR with limitations or weights on the frequency and time domain is
investigated in, e.g.,~\cite{BeaBG15,morGoyR17,Hal92,Pet13,morPetL14,morSin17,Vui14}. As continuation of our work
in~\cite{morBenKS16} regarding frequency-limited BT, we consider in this paper the numerically efficient
realization of time-limited balanced truncation (TLBT) for large-scale systems.
%%%%%%%%%%%%%%%%%%%%%%%%%%%%%%%%%%%%%%%%%%%%%%%%%%%%%%%%%%%%%%%%%%%%%%%%%%%%%%%%
\subsection{Overview of this Article}
We start in Section~\ref{sec:tlbt} by reviewing the general concept of time-limited BT from~\cite{morGawJ90}, mainly for restrictions of the form
\eqref{morBT:approx_timelimT}. This includes the associated time-limited Gramians as well as the respective Lyapunov equations.
Similar to standard BT, executing TLBT for large systems heavily relies on how well the time-limited Gramians can be approximated by low-rank factorizations.
This issue is investigated in Section~\ref{sec:decay} where we have a particular interest in the question when $t_e$
induces significant differences between the infinite and time-limited Gramians.  The actual issue of numerically dealing with the arising matrix functions and
computing the low-rank factors of the time-limited Gramians is topic of Sections~\ref{sec:expm} and \ref{sec:lr-gram}, respectively. Motivated by the
promising results for frequency-limited
BT studied in~\cite{morBenKS16}, we again employ rational Krylov subspace methods for this task. 
% However, because the occurring matrix
% function in TLBT is the well investigated matrix exponential, several other large-scale algorithms are applicable as well.
Section~\ref{sec:various} collects
different generalizations of TLBT including general state-space and certain differential algebraic systems, the time restriction
\eqref{morBT:approx_timelimT12}, and stability preservation.
In Section~\ref{sec:numex}, the proposed numerical approach is tested numerically with respect to the approximation of the Gramians as well as the
reduction of the dynamical
system. 
%%%%%%%%%%%%%%%%%%%%%%%%%%%%%%%%%%%%%%%%%%%%%%%%%%%%%%%%%%%%%%%%%%%%%%%%%%%%%%%%
\subsection{Notation}
The
real and complex numbers are denoted by, respectively, $\R$ and $\C$,
 $\R_-,~(\R_+)$, $\C_-~(\C_+)$ are the sets of
strictly negative (positive) real numbers and the open left (right) half plane.
The space of real (complex) Matrices of dimension $n\times m$
is $\Rnm~(\Cnm)$. For any complex quantity
$X=\Real{X}+\imath\Imag{X}$, we denote by $\Real{X}$ and $\Imag{X}$ its real and, respectively,
imaginary parts with $\imath$ being the imaginary unit, and its  
the complex conjugate is $\overline{X}=\Real{X}-\imath\Imag{X}$. 
The absolute value of any real or complex scalar is denoted by $|z|$. 
% and $\arg{z}$ is the argument.
Unless stated otherwise, $\|\cdot\|$ is the Euclidean vector- or subordinate
matrix norm (spectral norm). By $A^T$ and $A^H=\overline{A}^T$ we indicate the transpose and
 complex
conjugate transpose of a real and complex matrix, respectively. 
If $A\in\Cnn$ is a nonsingular, its inverse is $A^{-1}$, and $A^{-H}=(A^H)^{-1}$. 
Expressions of the form $x=A^{-1}b$ are always to be understood as solving a linear system $Ax=b$ for $x$.
The identity matrix of
dimension $n$ is indicated by $I_n$, and the vector of ones is denoted by $\mathbf{1}_m:=(1,\ldots,1)^T\in\Rm$.
The notation $A\succ0$ ($\prec0$) indicates 
symmetric positive (negative) definiteness of a symmetric or Hermitian 
matrix $A$, $\succeq$ ($\preceq$) refers to semi-definiteness, and $A\succeq(\preceq) B$ means $A-B\succeq(\preceq)0$.
The spectrum of a matrix $A$ is denoted by $\Lambda(A)$. 

%%%%%%%%%%%%%%%%%%%%%%%%%%%%%%%%%%%%%%%%%%%%%%%%%%%%%%%%%%%%%%%%%%%%%%%%%%%%%%%%
%%%%%%%%%%%%%%%%%%%%%%%%%%%%%%%%%%%%%%%%%%%%%%%%%%%%%%%%%%%%%%%%%%%%%%%%%%% 5
\section{Gramians and Balanced Truncation for Finite Time Horizons}\label{sec:tlbt}
% \subsection{The Finite-Time Gramians}
Since $A$ is assumed to be Hurwitz, the infinite Gramians $P_{\infty},~Q_{\infty}$~\eqref{morBT:gramians} can be represented in integral form as
% \begin{subequations}
\begin{align}\label{morBTlim:gram_int}
  P_{\infty}&=\intab{0}{\infty}
  \expm{At}BB^T\expm{A^Tt}\mathrm{d}t,\quad
  Q_{\infty}=\intab{0}{\infty}
  \expm{A^Tt}C^TC\expm{At}\mathrm{d}t.
\end{align}
% \end{subequations}
Restricting the integration limits in the integrals~\eqref{morBTlim:gram_int} to
a time interval $[t_s,t_e]$ immediately yields the 
\textit{time-limited Gramians} $P_{\cT}$, $Q_{\cT}$.
\begin{definition}[Time-limited Gramians~\cite{morGawJ90}]\label{def:TLgram}
 The time-limited reachability and
  observability Gramians of~\eqref{morBT:sys} with respect to the time-interval $\cT=[t_s,t_{e}],~0\leq t_e<t_{e}<\infty$ are defined by
  % \begin{subequations}
  \begin{align}\label{morBTlim:gramT_int}
     P_{\cT}&=\intab{t_s}{t_{e}}
  \expm{At}BB^T\expm{A^Tt}\mathrm{d}t,\quad
  Q_{\cT}=\intab{t_s}{t_{e}}
  \expm{\cT}C^TC\expm{At}\mathrm{d}t.
  \end{align}
  % \end{subequations}
\end{definition}
\begin{theorem}[Lyapunov equations for the time-limited Gramians~\cite{morGawJ90}]\label{thm:timelimgcale}
 The time-limited Gramians
  $P_{\cT}$ and $Q_{\cT}$ according to Definition~\ref{def:TLgram} are equivalently given in the following ways.
  \begin{enumerate}
  \item The finite time Gramians $P_{\cT},~Q_{\cT}$ from~\eqref{morBTlim:gramT_int} are given by
    \begin{subequations}\label{morBTlim:gramT}
    \begin{align}\label{morBTlim:gramlim}
      P_{\cT}&=\expm{At_s}P_{\infty}\expm{A^Tt_s}-\expm{At_{e}}P_{\infty}\expm{A^Tt_{e}},\\ %\\\label{morBTlim:gramlim}
      Q_{\cT}&=\expm{A^Tt_s}Q_{\infty}\expm{At_s}-\expm{A^Tt_{e}}Q_{\infty}\expm{At_{e}},
    \end{align}
    where $P_{\infty}$ and $Q_{\infty}$ are the infinite reachability and observability Gramians~\eqref{morBT:gramians}.
 \end{subequations}
  \item The time-limited Gramians $P_{\cT},~Q_{\cT}$ satisfy the
    \textit{time-limited reachability and observability Lyapunov equations}
    \begin{subequations}\label{morBTlim:cale}
      \begin{align}\label{calePT}
        AP_{\cT}+P_{\cT}A^T&=-B_{t_s}B_{t_s}^T+B_{t_{e}}B_{t_{e}}^T,\quad B_{t_s}:=\expm{At_s}B,~B_{t_e}:=\expm{At_e}B\\
        \label{caleQT}
        A^TQ_{\cT}+Q_{\cT}A&=-C_{t_s}^TC_{t_s}+C_{t_{e}}^TC_{t_{e}},\quad C_{t_s}:=C\expm{At_s},C_{t_e}:=C\expm{At_e}.
      \end{align}
    \end{subequations}
  \end{enumerate}
\end{theorem}
Note that the time-limited Gramians~\eqref{morBTlim:gramT_int} also exist for unstable $A$ and if $\Lambda(A)\cap\Lambda(-A)=\emptyset$ they
still solve the Lyapunov equations~\eqref{morBTlim:cale}, see~\cite{morKueR17}. Hence, TLBT might be one possible approach to reduce unstable systems. Except
for one short numerical experiment, we will not pursue this topic any further.
In analogy to the infinite time horizon case, the eigenvalues of the product $P_{\cT}Q_{\cT}$ are called \textit{time-limited Hankel singular values}. A basic
calculation shows
that the time-limited Hankel singular values are, as the standard Hankel singular values, invariant with respect to state-space transformations.
By comparing the Lyapunov equations for the infinite Gramians~\eqref{morBT:gramians} with~\eqref{morBTlim:cale}, one immediately sees that the only differences
are the
inhomogeneities, while the left hand sides are the same unchanged (adjoint) Lyapunov operators. This raises the question how much different the time-limited
Gramians are from the infinite ones, and how this depends on the  time interval of interest. We will pursue this in the next section, especially
regarding the numerically important issue of how well the time-limited Gramians can be approximated by low-rank factorizations $P_{\cT}\approx
Z_{P_{\cT}}Z_{P_{\cT}}^T$, $Q_{\cT}\approx Z_{Q_{\cT}}Z_{Q_{\cT}}^T$. Of course, before approximately solving the time-limited Lyapunov equations, 
the actions of the matrix exponentials $\expm{At_i}$ to $B$ (as well as $\expm{A^Tt_i}$ to $C^T$) have to be dealt with numerically. 
Numerical approaches for handling the matrix exponential and computing low-rank solution factors $Z_{P_{\cT}},~Z_{Q_{\cT}}$ are topic of
Section~\ref{sec:expmLR}.

A square-root version of TLBT is simply carried out by substituting Step~\ref{alg:bt_lyap} of
Algorithm~\ref{alg:lrsrbt} by the code snippet shown in Algorithm~\ref{alg:lrsrtlbt} below%
\LinesNumberedHidden
\begin{algorithm}[h]
 \caption{Required changes in Algorithm~\ref{alg:lrsrbt} for TLBT.}
  \label{alg:lrsrtlbt}
\nlset{1a}
  Compute (approximation of) $B_{t_i}:=\expm{At_i}B$, $C_{t_i}=C\expm{A^Tt_i}$, $i\in\lbrace s,e\rbrace$.\;
\nlset{1b}
  Compute $Z_{P_{\cT}},~Z_{Q_{\cT}}$, such  that $P_{\cT}\approx
Z_{P_{\cT}}Z_{P_{\cT}}^T$, $Q_{\cT}\approx Z_{Q_{\cT}}Z_{Q_{\cT}}^T$ in~\eqref{morBTlim:cale}.\;
\end{algorithm}
\LinesNumbered
and using the low-rank
solution factors $Z_{P_{\cT}}$, $Z_{Q_{\cT}}$ in the remaining steps. 
Depending on the time region of interest, TLBT might in general not be a stability preserving method and, thus, there is also no $\cH_{\infty}$-error
bound similar
to~\eqref{morBT:error}. This can be regained by modifying TLBT further~\cite{morGugA04} and we discuss this issue in
Section~\ref{sec:various}. Without this modification it is possible to establish an error bound in the $\cH_2$-norm~\cite{morKueR17}.
%%%%%%%%%%%%%%%%%%%%%%%%%%%%%%%%%%%%%%%%%%%%%%%%%%%%%%%%%%%%%%%%%%%%%%%%%%%%%%%%%%%%%%%%%%%%%%%%%%%%%%%%%%%%%%%%%%%
\section{Numerical Computation of Low-rank Factors of Time-Limited Gramians}\label{sec:expmLR}
This section is concerned with the actual numerical computation of low-rank factors of the time-limited Gramians.
Before algorithms for dealing with the matrix exponentials and the Lyapunov equations are discussed, we briefly investigate the numerical ranks of the
time-limited Gramians. 
For the sake of brevity, all considerations will be mostly restricted to the reachability Gramian because the observability Gramian can
be dealt with similarly by replacing $A,B$ with $A^T,C^T$, respectively. Moreover, only the situation~\eqref{morBT:approx_timelimT} will be discussed,
i.e., $t_s=0$ and $0<t_e<\infty$.
%%%%%%%%%%%%%%%%%%%%%%%%%%%%%%%%%%%%%%%%%%%%%%%%%%%%%%%%%%%%%%%%%%%%%%%%%%%%%%%%%%%%%
\subsection{The Difference of the Infinite and Time-limited Gramians}\label{sec:decay}
In order to approximate $P_{\cT}$, $Q_{\cT}$ by a low-rank factorizations, it is desirable that their eigenvalues decay rapidly. 
For investigating this decay we assume from now that eigenvalues of symmetric (positive definite) matrices are given in a non-increasing order
$\lambda_1\geq\ldots\geq\lambda_n$. The inhomogeneities of the time-limited Lyapunov equations~\eqref{morBTlim:cale} have up to twice the rank of their
unlimited counterparts~\eqref{morBT:gramians}. Hence, by the available theory on the eigenvalue decay of solutions of matrix
equations~\cite{AntSZ02,Gra04,Pen00,Sab07} one expects that the eigenvalues of the time-limited Gramians decay somewhat slower than those of the infinite ones.
We will see in the numerical experiments that, similar to the frequency-limited Gramians, in most situations it is the opposite case: the eigenvalues of the
time-limited Gramians exhibit a faster decay and, consequently, have smaller numerical ranks. Assuming that $(A,B)$ is controllable 
(rank$[A-\lambda I,B]=n,~\forall\lambda\in\C$)
it holds $P_{\infty}\succ0$
and, using~\eqref{morBTlim:gramT_int}, the relation \eqref{morBTlim:gramlim} yields $0\preceq P_{\cT}=P_{\infty}-\expm{At_{e}}P_{\infty}\expm{A^Tt_{e}}$. Since
$E(t_e):=\expm{At_{e}}P_{\infty}\expm{A^Tt_{e}}\succ0$ it holds
\begin{align*}
P_{\infty}\succeq P_{\cT}\quad\text{and also}\quad\lambda_1(P_{\infty})=\|P_{\infty}\|_2\succeq  \|P_{\cT}\|_2=\lambda_1(P_{\cT}).
\end{align*}
Due to the stability of $A$, the difference $E(t_e)=P_{\infty}-P_{\cT}$ will decay for increasing values
of $t_e$. Tight bounds for the eigenvalue behavior of $P_{\cT}$ and $E(t_e)$ with respect to the parameter $t_e$ are  difficult to derive and is a research
topic of its
own. Here, for simplicity we restrict the discussion to the case $m=1$ and present a basic investigation of how the decay of $E(t_e)$ depends on
$t_e$. 
\begin{lemma}
 Let $B\in\Rn$, $(A,B)$ controllable, $A$ be diagonalizable, i.e., $A=X\Lambda X^{-1}$,
$\Lambda=\diag{\lambda_1,\ldots,\lambda_n}$, and define
$w:=X^{-1}B$, $X_B:=X\diag{w}$, $N(t_e):=X\expm{\Lambda t_e}\diag{w}$.  Then $E(t_e)=\expm{At_{e}}P_{\infty}\expm{A^Tt_{e}}=N(t_e)\cC N(t_e)^H$, 
 \begin{align}\label{gramT_fact}
  P_{\cT}=P_{\infty}-E(t_e)=P_{\infty}-N(t_e)\cC N(t_e)^H=X_B\left(\cC-\expm{\Lambda t_e}\cC\expm{\Lambda^H
t_e}\right)X_B^H,
 \end{align}
and $\cC:=\left(\tfrac{-1}{\lambda_i+\overline{\lambda_j}}\right)_{i,j=1}^n$ is a Hermitian positive definite Cauchy matrix.
\end{lemma}
\begin{proof}
 Apply the eigendecomposition of $A$ and $P_{\infty}=X_B\cC X_B^H$ from~\cite[Lemma 3.2]{AntSZ02} to~\eqref{morBTlim:gramlim}.
\end{proof}
Consider the impulse response of~\eqref{morBT:sys},
\begin{align*}
 y_{\delta}(t)=C\eta(t),\quad \eta(t):=\expm{At}B=X\expm{\Lambda t}w=N(t)\mathbf{1}_n,
\end{align*}
indicating that  the \textit{impulse-to-state-map} $\eta(t)$ and $N(t)$ decay at a similar rate. With the \textit{spectral abscissa}
$R:=\max\limits_{\lambda\in\Lambda(A)}\Real{\lambda}$, the basic point wise bounds 
\begin{align*}
 \|\eta(t)\|\leq \expm{Rt}\|X\|\|w\|,\quad \|N(t)\|\leq \expm{Rt}\|X\|\|w\|_{\infty}
\end{align*}
make this more visible. Using this and~\eqref{gramT_fact}, we can conclude that for increasing~$t$, $E(t)$ is getting smaller
at a similar speed as $\eta(t)$. Consequently, a significant difference between $P_{\cT}$ and $P_{\infty}$ might be only observed when $t_e$ is chosen small
enough in the sense that $\eta(t_e)$ has not reached an almost stationary state. The handling of the case $m>1$ can be carried out similarly.
Moreover, with a similar argumentation the decay rate of the time-limited Hankel singular values can be roughly connected to the decay of $y_{\delta}(t)$. To
conclude, TLBT might be only practicable for small time horizons or for weakly damped systems. 
A similar investigation regarding TLBT for unstable systems is given 
in~\cite{morSin17}.
%%%%%%%%%%%%%%%%%%%%%%%%%%%%%%%%%%%%%%%%%%%%%%%%%%%%%%%%%%%%%%%%%%%%%%%%%%%%%%%%%%%%%
\subsection{Approximating the Products with the Matrix Exponential}\label{sec:expm}
There are several numerical approaches available to approximate the action of the matrix exponential to (a couple of) vectors, see, e.g.,
\cite{AlmH11,BecR09,CalKOetal14,DavH05,DruKZ09,DruLZ10,FroLS17,FroS08,Gue10,Gue13,Hig08,Kni92,Saa92a}. Here, we are mostly interested in projection methods
using (block)
rational
Krylov subspaces of the form
\begin{align}\label{rks}
 \cR\cK_k=\myspan{q_1,\ldots,q_k},\quad q_k=\left[\prod\limits_{j=1}^k (A-s_jI)^{-1}\right]B
\end{align}
where $s_k\in\C_+\cup\imath\R\cup\lbrace{\infty\rbrace}$ are called shifts. They represent the poles of a rational approximation
$r_k=\psi_{k-1}/\phi_{k-1}$ of
$\mathrm{e}^{z}$ with polynomials $\psi_{k-1},\phi_{k-1}$ of degree at most $k-1$.
Let $Q_k\in\R^{n\times km}$ have orthonormal columns and
$\range{Q_k}=\cR\cK_k$. Then a Galerkin approximation~\cite{Saa92a} of
$\expm{At_{e}}B$ takes the
form
\begin{align}\label{rk_expmv}
 B_{t_e}\approx B_{t_e,k}&:=Q_k\hB_{t_e,k},\quad \hB_{t_e,k}:=\expm{H_kt_e}B_k,\\\nonumber
 H_k&:=Q_k^TAQ_k,\quad B_k:=Q_k^TB.
\end{align}
Note that for $m=1$, the rational function $r_k$ underlying the rational Krylov approximation $Q_k\left(\expm{H_kt_e}\right)Q_k^TB$ interpolates
$f(z)=\mathrm{e}^{zt}$ at $\Lambda(H_k)$ (rational Ritz values)~\cite[Theorem~3.3]{Gue13}. Further information on the approximation properties can be found in,
e.g.,~\cite{BecR09,BerG15,DruKZ09,DruLZ10,DruS11,Gue10}.
In the remainder we assume $s_1=\infty$ s.t. $q_1=B$, $\range{B}\subseteq\range{Q_k}$ and $B_k=[\beta^T,0]^T\in\R^{km\times m}$, where $\beta\in\R^{m\times m}$.
The orthonormal basis matrix $Q_k$ and the restriction $H_k$ can be efficiently computed by a (block) rational Arnoldi process~\cite{Ruh84}. Since $H_k$ is
low-dimensional, $\expm{H_kt_e}$ can be computed by standard dense methods for the matrix exponential~\cite{Hig08}. The choice of shifts $s_k$  ($k>1$) is
crucial for a rapid convergence and several strategies exists for this purpose~\cite{Gue13}. In this work we exclusively use adaptive shift generation
techniques~\cite{DruKZ09,DruLZ10,DruS11} because this appeared to be the safest strategy in the majority of experiments. For the situation $m=1$ and after step
$k$ of the rational Arnoldi process, the next shift $s_{k+1}$ is selected via
\begin{align}\label{rk_adapt}
 s_{k+1}=\argmax_{\partial\bS_k} |r_k(s)|,\quad r_k(s)=\prod\limits_{j=1}^k\tfrac{s-z_j}{s-s_j},
\end{align}
where $z_j$ are the eigenvalues of $H_k$ (Ritz values of $A$); $s_j, j=1,\ldots,k$ are the previously used shifts; and $\partial\bS_k$ marks a set of discrete
points
from the boundary of $\bS_k$ approximating $\Lambda(A)$. We follow~\cite{DruS11} and use $\bS_k$ as the convex hull of $\Lambda(H_k)$. In the symmetric case,
one can also use the spectral interval given by the largest and smallest eigenvalue of $A$~\cite{DruKZ09,DruLZ10,Gue13}. 
For $m>1$ we simply use each $s_j$ in the denominator of $r_k$ in~\eqref{rk_adapt} $m$ times as in~\cite{DruS11}. A different technique to deal with $m>1$
includes, e.g., tangential rational Krylov methods~\cite{DruSZ14}.  
% or, in the case of polynomial Arnoldi approximations ($s_j=\infty$), approaches based on
% different inner products~\cite{FroLS17}. 
The rational Krylov subspace method was chosen not only because of the good approximation properties, but also because the generated subspace is also a good
candidate to acquire low-rank solution factors of~\eqref{morBTlim:cale}.
%%%%%%%%%%%%%%%%%%%%%%%%%%%%%%%%%%
\subsection{Computing the low-rank Gramian factors}\label{sec:lr-gram}
Using~\eqref{rks},~\eqref{rk_expmv}, a Galerkin approximation of the time-limited Gramian is $P_{\cT,k}=Q_kY_kQ_k^T\approx P_{\cT}$,
where $Y_k$ solves the projected time-limited Lyapunov equation 
\begin{align}
 H_kY_k+Y_kH_k^T=-B_kB_k^T+\hB_{t_e,k}\hB_{t_e,k}^T,
\end{align}
(cf.~\cite{Saa90}). 
It holds $B_k=\hB_{0,k}=e^{H_k t_s}B_k$ for the special case $t_s=0$ considered here, indicating already how $t_s\neq 0$ can be
included
as well.
Hence, after $B_{t_e,k}$ of sufficient accuracy is found, we follow the idea in~\cite{morBenKS16} by recycling the rational Krylov basis and
continuing 
the rational Arnoldi process for~\eqref{calePT} until $P_{\cT,k}$ leads to a sufficiently small norm of the Lyapunov residual. For the Lyapunov stage, the
same adaptive shift generation~\eqref{rk_adapt} can be used~\cite{DruS11}. The rational Krylov subspace method for generating a low-rank
approximation $Z_kZ_k^T\approx P_{\cT}$ is illustrated in Algorithm~\ref{alg:ksm_tlimcale}. 

\begin{algorithm}[t]
  \SetEndCharOfAlgoLine{} \SetKwInOut{Input}{Input}\SetKwInOut{Output}{Output}
  \caption[Rational Krylov subspace method for time-limited CALEs]{Rational Krylov subspace
    method for time-limited Lyapunov equations~\eqref{calePT}}
  \label{alg:ksm_tlimcale}
  % \begin{algorithmic}[1]
  \Input{$A,~B,~t_e$ as in~\eqref{calePT},
    tolerances $0<\tau_f,~\tau_P\ll1$.}
  \Output{$Z_{k}\in\R^{n\times \ell}$ such that
    $Z_{k}Z_k^T\approx P_{\cT}$ with
    $\ell\leq mk\ll n$.}
  $B=q_1\beta$ s.t. $q_1^Tq_1=I_m$, $Q_1=q_1$.\;\nllabel{ksm_ini} 
  \For{$k=1,2,\ldots$} {%
    Get new shift $s_{k+1}$ via, e.g.,~\eqref{rk_adapt}.\;    
    Solve $(A-s_{k+1}I)g=q_k$ for $g$.\nllabel{ksm_linsys}\;
    Real, orthogonal expansion of $Q_k$: 
    $g_+=\Real{g}-Q_k(Q_k^T\Real{g})$.\nllabel{ksm_exp1}\;
    $q_{k+1}=g_+\beta_k$ s.t. $q_{k+1}^Tq_{k+1}=I_m$, $Q_{k+1}=[Q_k,~q_{k+1}]$.\;
    \If{$\Imag{s_{k+1}}\neq 0$}
    {%
    $k=k+1$, $g_+=\Imag{g}-Q_k(Q_k^T\Imag{g})$\;
    $q_{j+1}=g_+\beta_k$ s.t. $q_{k+1}^Tq_{k+1}=I_m$, $Q_{k+1}=[Q_k,~q_{k+1}]$.\nllabel{ksm_exp2}\; 
    }
    $H_k=Q_k^TAQ_k$, $B_k=Q_k^TB$.\nllabel{ksm_Hj}\;
    $\hB_{t_e,k}=\expm{H_kt_e}B_k$, $B_{t_e,k}=Q_k\hB_{t_e,k}$\nllabel{ksm_Bj}\;
    % }
    \If{$\|B_{t_e,k}-B_{t_e,k-1}\|/\|B_{t_e,k}\|<\tau_f$\nllabel{ksm_Bj_test}}{%
    Solve
      $H_kY_k+Y_kH_k^T+B_kB_k^T-\hB_{t_e,k}\hB_{t_e,k}^T=0$ for $Y_{k}$.\nllabel{ksm_Pj}\;
      Set
$\mu_k:=\tfrac{\|A(Q_kY_kQ_k^T)+(Q_kY_kQ_k^T)A^T+BB^T-B_{t_e,k}B_{t_e,k}^T\|}{\|B_kB_k^T-\hB_{t_e,k}\hB_{t_e,k}^T\|}$.\nllabel{ksm_Lj}\;
      \If{$\mu_k<\tau_P$\nllabel{ksm_Lj_test}} {%
        $Y_{k}=S\Gamma S^T$, $S^TS=I_{km}$, $\Gamma=\diag{\gamma_1,\ldots,\gamma_{km}}$.\nllabel{ksm_Zj1}\;
       Truncate if necessary: $\Gamma=\diag{\gamma_1,\ldots,\gamma_{\ell}}$, $S=S(:,1:\ell)$, $\ell\leq mk$.\;
        \Return low-rank solution factor
        $Z_{k}=Q_kS\Gamma^{\half}$.\nllabel{ksm_Zj2} \;
%         Stop rational Arnoldi process.\;
        } 
        } 
    }
\end{algorithm}
In the presence of a complex shift $s_{k+1}$, it is implicitly assumed that $\overline{s_{k+1}}$ is the subsequent shift. For this situation, the 
orthogonal expansion in Steps~\ref{ksm_exp1}--\ref{ksm_exp2} of the already computed basis matrix $Q_k$ by real basis vectors goes back to~\cite{Ruh94c}.
The projected matrix $H_k$ in Step~\ref{ksm_Hj}, as well as the Lyapunov residual norm in Step~\ref{ksm_Lj} can be computed without accessing the large matrix
$A$, see, e.g.,~\cite{BerG15,DruS11,Gue13,Ruh84}. 
The small Lyapunov equation in Step~\ref{ksm_Pj} can be solved by standard methods for dense matrix equations, e.g., the  Bartels-Stewart~\cite{BarS72}
method
which we employ here.
Once the scaled Lyapunov residual norm falls below the desired threshold, the rational Arnoldi process is
terminated and the Steps~\ref{ksm_Zj1}--\ref{ksm_Zj2} bring the computed low-rank Gramian approximation in the desired form $Z_{k}Z_k^T$ and allow a rank
truncation. Note that typically, once the approximation of $B_{t_e}$ is found, the generated subspace is already good enough to acquire a low-rank Gramian
approximation without many additional iterations of the rational Arnoldi process. Often, the criteria in Steps~\ref{ksm_Bj_test},\ref{ksm_Lj_test} are satisfied
in the same iteration step.
%%%%%%%%%%%%%%%%%%%%%%%%%%%%%%%%%%%%%%%%%%%%%%%%%%%%%%%%%%%%%%%%%%%%%%%%%%%%%%%%%%%%%%%%%%%%%%%
\section{Extensions and Further Problems}\label{sec:various}
%%%%%%%%%%%%%%%%%%%%%%%%%%%%%%%%%%%%%%%%%%%%%%%%%%%%%%%%%%%%%%%%%%%%%%%%%%%%%%%%%%%%%%%%%%%%%%%
\subsection{Multiple time values}\label{ssec:multitime}
Computing low-rank factors of the time-limited Gramians~\eqref{morBTlim:cale} for the more general approximation setting~\eqref{morBT:approx_timelimT12} with a
nonzero start time $t_s$, i.e., $\cT=[t_s,t_e]$, can also be done using Algorithm~\ref{alg:ksm_tlimcale} with minor adjustments. Having computed a rational
Krylov basis for
approximating $\expm{At_e}B$ for some time value $t_e$, the same basis typically also provides good approximations for any other time values~\cite{Gue13}.
Consequently, Algorithm~\ref{alg:ksm_tlimcale} has to be simply changed by adding $\hB_{t_s,k}=\expm{H_kt_s}B_k$, $B_{t_s,k}=Q_k\hB_{t_s,k}$
and appropriately adjusting the steps regarding the Gramian approximation.
Also the methods relying on Taylor approximations~\cite{AlmH11} can be easily modified to handle multiple time values.
However, even if a nonzero $t_s$ does not yield additional computational complications, this does not imply that TLBT will produce a accurate approximation of
the transient behavior of~\eqref{morBT:sys} in $[t_s,t_e]$. Already the original TLBT paper~\cite{morGawJ90} states that TLBT in $[t_s,t_e]$ is expected to
only give good approximations of the impulse response ($u(t)=v\delta(t)$, $v\in\Rm$) $y_{\delta}(t)=C\expm{At}Bv$
% \begin{align*}
%  
% \end{align*}
and numerical experiments confirm this. 
For an intuitive explanation assume for simplicity that no truncation is done in TLBT, i.e. in 
Step~\ref{alg:BT_firstord_svd} of Algorithm~\ref{alg:lrsrbt}. Then it holds
$\range{Q_k}=\range{T}$, $\expm{At_s}B\approx B_{k,t_s}\in\range{Q_k}$,
$\expm{At_e}B\approx B_{k,t_e}\in\range{Q_k}$, 
and likewise for $C\expm{At_s}$, $C\expm{At_s}$ and $\range{S}$. Hence, the impulse response is accurately approximated at the relevant times. 

For the response of an arbitrary
 input $u(t)$, $y_{u}(t)=C\intab{0}{t}\expm{A(t-\tau)}Bu(\tau)\mathrm{d}\tau$, such argumentation clearly does not automatically hold.
 The value of $y_{u}(t)$ with respect to a general $u(t)$ depends, in general, on the values at the times before $t$.
 In the present form TLBT restricts the approximation to the time frame $\cT$ and, thus, $y_{u}(t)$ will be poorly approximated for $t\leq t_s$ which, in
turn, makes it difficult to acquire good approximations in $\cT$.
 
One approach is to apply a time translation to the underlying system~\eqref{morBT:sys} such that the requested time-interval is
transformed to $[0,t_e-t_s]$.
However, this time translation will also introduce an inhomogeneous initial value $x_0$, which is an additional difficulty for model order reduction. Some
strategies to cope with nonzero initial values are given in~\cite{morBeaGM17,morHeiRA11} and we plan to investigate the incorporation to time-limited BT in the
future.
%%%%%%%%%%%%%%%%%%%%%%%%%%%%%%%%%%%%%%%%%%%%%%%%%%%%%%%%%%%%%%%%%%%%%%%%%%%%%%%%%%%%%%%%%%%%%%%
\subsection{Generalized state-space and index-one descriptor systems}\label{ssec:gen}
Consider generalized state-space systems 
\begin{subequations}\label{morBT:gsys}
  \begin{align}
    M\dot x(t)&=Ax(t)+Bu(t),\quad x(0)=0,\\
    y(t)&=Cx(t)
  \end{align}
\end{subequations}
with $M\in\Rnn$ nonsingular. Simple manipulations reveal that, similar to unrestricted~\cite{Ben04} and frequency-limited BT~\cite{morBenKS16},
the time-limited Gramians of~\eqref{morBT:gsys} are $P_{\cT}$, $M^TQ_{\cT}M$, where $P_{\cT}$, $Q_{\cT}$ solve the time-limited generalized Lyapunov equations
   \begin{subequations}\label{morBTlim:gcale}
    \begin{align}\label{gcalePT}
        AP_{\cT}M^T+MP_{\cT}A^T&=-B_{t_s}B_{t_s}^T+B_{t_{e}}B_{t_{e}}^T,~B_{t_i}:=M\expm{M^{-1}At_i}M^{-1}B,\\
        \label{gcaleQT}
        A^TQ_{\cT}M+M^TQ_{\cT}A&=-C_{t_s}^TC_{t_s}+C_{t_{e}}^TC_{t_{e}},~ C_{t_i}:=C\expm{M^{-1}At_i}
\end{align}
   \end{subequations}
for $t_i=t_s,t_e$. Note that $B_{t_i}:=\expm{AM^{-1}t_i}B$. Low-rank factors of $P_{\cT}$, $Q_{\cT}$ can be computed by methods for
generalized
Lyapunov equations.
In particular, the rational Krylov subspace methods which we employ for approximating $B_{t_i}$ will implicitly work on $M^{-1}A$, $M^{-1}B$ or
alternatively, if $M=LL^T\succ 0$, on $L^{-1}AL^{-T}$, $L^{-1}B$. With low-rank approximations $P_{\cT}\approx Z_{P_{\cT}}Z_{P_{\cT}}^T $, $Q_{\cT}\approx
Z_{Q_{\cT}}Z_{Q_{\cT}}^T$, the SVD of $Z_{Q_{\cT}}^TMZ_{P_{\cT}}$ has to be used in Step~\ref{alg:BT_firstord_svd} of Algorithm~\ref{alg:lrsrbt}.

\smallskip
In some of the numerical experiments we will encounter the situation\\
$M=\smb M_1&0\\0&0\sme,$ $A=\smb A_1&A_2\\A_3&A_4\sme,$ $B=\smb B_1\\B_2\sme,$ $C=\smb C_1&C_2\sme$
with $M_1\in\R^{n_f\times n_f},~A_4\in\R^{n-n_f\times n-n_f}$ nonsingular, s.t.~\eqref{morBT:gsys} becomes a semi-explicit index-one descriptor system.
Eliminating the algebraic constraints leads to a general state-space system defined by $M_1$, 
$\hA:=A_{1}-A_{2}A_{4}^{-1}A_{3}$ and $\hB:=B_1-A_{2}A_{4}^{-1}B_2$, $\hC:=C_1-C_{2}A_{4}^{-1}A_3$, and an additional feed-through term $-C_{2}A_{4}^{-1}B_2$,
see~\cite{morFreRM08}.
Time-limited and unrestricted BT can be applied right away to this system via~\eqref{morBTlim:gcale} defined by $M_1,\hA,\hB,\hC$.
However, the matrix $\hA$ will in general be dense and, thus, solving the linear systems in Algorithm~\ref{alg:ksm_tlimcale} can be very expensive. In the
context of unrestricted BT, the authors of~\cite{morFreRM08} exploit
in a LR-ADI iteration for the Gramians that the arising dense linear systems $(\hA-sM_1)\hV=\hW$ are equivalent to the sparse linear systems $(A-sM)V=W$ with
$V=[\hV^T,\Psi]^T$,~$W=[\hW^T,0]^T$ which are easier to solve numerically. We use the same trick within Step~\ref{ksm_linsys} of
Algorithm~\ref{alg:ksm_tlimcale}.
%%%%%%%%%%%%%%%%%%%%%%%%%%%%%%%%%%%%%%%%%%%%%%%%%%%%%%%%%%%%%%%%%%%%%%%%%%%%%%%%%%%%%%%%%%%%%%%
\subsection{Stability preservation and modified TLBT}\label{ssec:modBT}
Because of the altered and sometimes indefinite right hand sides of \eqref{morBTlim:cale},\eqref{morBTlim:gcale}, TLBT is in general not stability preserving.
Only when the used time interval is long enough such that the right hand sides are negative semi-definite, TLBT will produce an asymptotically stable reduced
order
model~\cite[Condition~1]{morGawJ90}. For the general situation, a stability preserving modification of TLBT is proposed in~\cite{morGugA04} using
the Lyapunov equations
    \begin{align}\label{morBTlim:gcalemod}
    \begin{split}
      AP^{\text{mod}}_{\cT}M^T+MP^{\text{mod}}_{\cT}A^T&=-B_{\text{mod}}B^T_{\text{mod}},\\
        A^TQ^{\text{mod}}_{\cT}M+M^TQ^{\text{mod}}_{\cT}A&=-C^T_{\text{mod}}C_{\text{mod}},
    \end{split}\\\nonumber
  B_{\text{mod}}:=Q_B\diag{|\lambda^B_1|,\ldots,|\lambda^B_{r_B}|}^\half,&\quad C_{\text{mod}}:=\diag{|\lambda^C_1|,\ldots,|\lambda^C_{r_C}|}^\half
Q_{C}^T.
\end{align}
% with $B_{\text{mod}}:=Q_B(:,1:r_B)\diag{|\lambda^B_1|,\ldots,|\lambda^B_{r_B}|}^\half$, $C_{\text{mod}}:=\diag{|\lambda^C_1|,\ldots,|\lambda^C_{r_C}|}^\half
% Q_{C}(:,1:r_C)^T$, 
where $Q_B\in\R^{n\times r_B}$, $Q_C\in\R^{n\times r_C}$ contain the eigenvectors corresponding to the $r_B\leq 2m,r_C\leq 2p$ nonzero eigenvalues $\lambda^B_i,
\lambda^C_i$
of the right hand sides of~\eqref{calePT},\eqref{gcalePT} and, respectively,~\eqref{caleQT},\eqref{gcaleQT}. 
The rational Krylov subspace method in Algorithm~\ref{alg:ksm_tlimcale} for $M=I, t_s=0$ can be easily modified for \eqref{morBTlim:gcalemod} by replacing
Step~\ref{ksm_Pj} with the steps shown in Algorithm~\ref{alg:modBT}.
\LinesNumberedHidden
\begin{algorithm}[h]
% \renewcommand{\thealgocf}{2}
% \SetAlgorithmName{Code Snippet}
\SetEndCharOfAlgoLine{}
  \caption{Changes in Algorithm~\ref{alg:ksm_tlimcale} for modified time-limited Gramians}
  \label{alg:modBT}
\nlset{13a}
  Compute partial eigendecomposition 
% $(BB^T-B_{t_e,j}B_{t_e,j}^T)Q_B=Q_B\diag{\lambda^B_1,\ldots,\lambda^B_{r_B}}$ with $Q_B^TQ_B=I_{r_B}$, $\lambda^B_i\neq 0$.\;
$\left(\smb\beta\beta^T&0\\0&0\sme-\hB_{t_e,k}\hB_{t_e,k}^T\right)Q_B=Q_B\diag{\lambda^B_1,\ldots,\lambda^B_{r_B}}$, $Q_B^TQ_B=I_{r_B}$, $\lambda^B_i\neq
0$.\;
\nlset{13b}
 Factor of projected modified rhs $B_{\text{mod},k}:= Q_B\diag{|\lambda^B_1|,\ldots,|\lambda^B_{r_B}|}^\half$.\;
 \nlset{13c}
Solve $H_kY_k+Y_kH_k^T+B_{\text{mod},k}B_{\text{mod},k}^T$ for $Y_{k}$.\;
\end{algorithm}
\LinesNumbered
Generalization for $M\neq I$ and $0<t_s<t_e$ are straightforward.
Note that because the modified time-limited Gramians do not fulfill a relation of the form \eqref{morBTlim:gramT_int} or \eqref{morBTlim:gramT}, we cannot
expect a fast singular value decay similar to $P_{\cT},~Q_{\cT}$. 
As observed in~\cite{morBenKS16,morGugA04} for modified frequency-limited
BT, modified TLBT might also lead to less accurate approximations in the considered time region compared to unmodified TLBT.
%    \end{subequations}
%%%%%%%%%%%%%%%%%%%%%%%%%%%%%%%%%%%%%%%%%%%%%%%%%%%%%%%%%%%%%%%%%%%%%%%%%%%%%%%%%%%%%%%%%%%%%%%
\section{Numerical Experiments}\label{sec:numex}
Here, we illustrate numerically the results of Section~\ref{sec:expmLR} regarding
the numerical rank of the frequency-limited Gramians as well as the numerical method for computing low-rank factors of $P_{\cT},~Q_{\cT}$.
Afterwards, the quality of the approximations obtained by TLBT with the low-rank factors is evaluated and compared against unlimited BT.
Further topics like nonzero starting times and the modified TLBT scheme are also examined along the way. 
All experiments are done in \matlab~8.0.0.783 (R2012b) on a \intel\xeon CPU
X5650 @ 2.67GHz with 48 GB RAM. 
Table~\ref{tab:examples} list the used examples and their properties. For time-domain simulation of~\eqref{morBT:sys},~\eqref{morBT:gsys} and the reduced order
models for a
given input function $u(t)$, an implicit midpoint rule with a fixed small time step $\Delta t$ is used. 
Because the impulse response 
of~\eqref{morBT:sys} for $u(t)=\delta(t)v$, $v\in\Rm$, $x_0=0$ can be expressed as $y_{\delta}(t) = C\expm{At}Bv$, it is computed
by the same integrator applied to the uncontrolled system ($u(t)=0$) with initial condition $x_0=Bv$.
For
the impulse, or step responses, the vector distributing the control to the columns of $B$ is set to $v=\mathbf{1}_m$. 
\begin{table}[t]
  \footnotesize
  \centering
  \caption{Dimensions, final integration time $t_f$, step size $\Delta t$, matrix
    properties, and source of the test
    systems.}
\setlength{\tabcolsep}{0.5em}
  \begin{tabularx}{\textwidth}{|l|c|c|c|c|c|l|X|}
    \hline
    Example & $n$ & $m$&$p$&$t_f$&$\Delta t$&properties&source\\
    \hline%\hline
    \texttt{bips\_606}&7135&4&4&20&0.04&index-1, $n_f=606$, $M_1=I_{n_f}$&morwiki\footnotemark[1],~\cite{morFreRM08}\\
    \texttt{bips\_3078}&21128&4&4&20&0.04&index-1, $n_f=3078$, $M_1=I_{n_f}$&morwiki\footnotemark[1],~\cite{morFreRM08}\\
    \texttt{vertstand}&16626&6&6&600&0.6&$A\prec0$, $M\succ0$, $C$ random&morwiki\footnotemark[1],~\cite{GroSGetal12}\\
    \texttt{rail}&79841&7&6&400&0.4&$A\prec0$, $M\succ0$&Oberwolfach Collection\footnotemark[2], ID=38881\\
    \hline
  \end{tabularx}\label{tab:examples}
\end{table}

Because the \texttt{bips} systems have eigenvalues very close to the imaginary axis which caused difficulties for all considered numerical methods, 
the shifted matrix $A-0.08 M$ is used instead as in~\cite{morFreRM08}. The generalized state-space systems \texttt{vertstand} and
\texttt{rail} are dealt with Cholesky factorizations $M=LL^T$ as explained in Section~\ref{ssec:gen}. 
There, the sparse Cholesky factors $L$ are
computed by the \matlab~command \texttt{chol(M,'vector')}. Matrix exponentials and Lyapunov equations defined by smaller dense matrices (including the
projected ones in Algorithm~\ref{alg:ksm_tlimcale}) are solved directly by the \texttt{expm} and \texttt{lyap} routines.
\footnotetext[1]{\url{https://morwiki.mpi-magdeburg.mpg.de/morwiki}}
\footnotetext[2]{\url{http://portal.uni-freiburg.de/imteksimulation/downloads/benchmark}}
%%%%%%%%%%%%%%%%%%%%%%%%%%%%%%%%%%%%%%%%%%%%%%%%%%%%%%%%%%%%
\subsection{Decay of the eigenvalues of the Gramians and the time-limited singular values}
\begin{figure}[t]
  \centering
% \tikzsetnextfilename{figure1}
% \begin{tikzpicture}
%  \input{./images/bips606_svd}
% \input{./images/bips606_impulsrep}
% \input{./images/bips606_hsvd}
% \input{./images/bips606_numranks}
% \end{tikzpicture}
\includegraphics{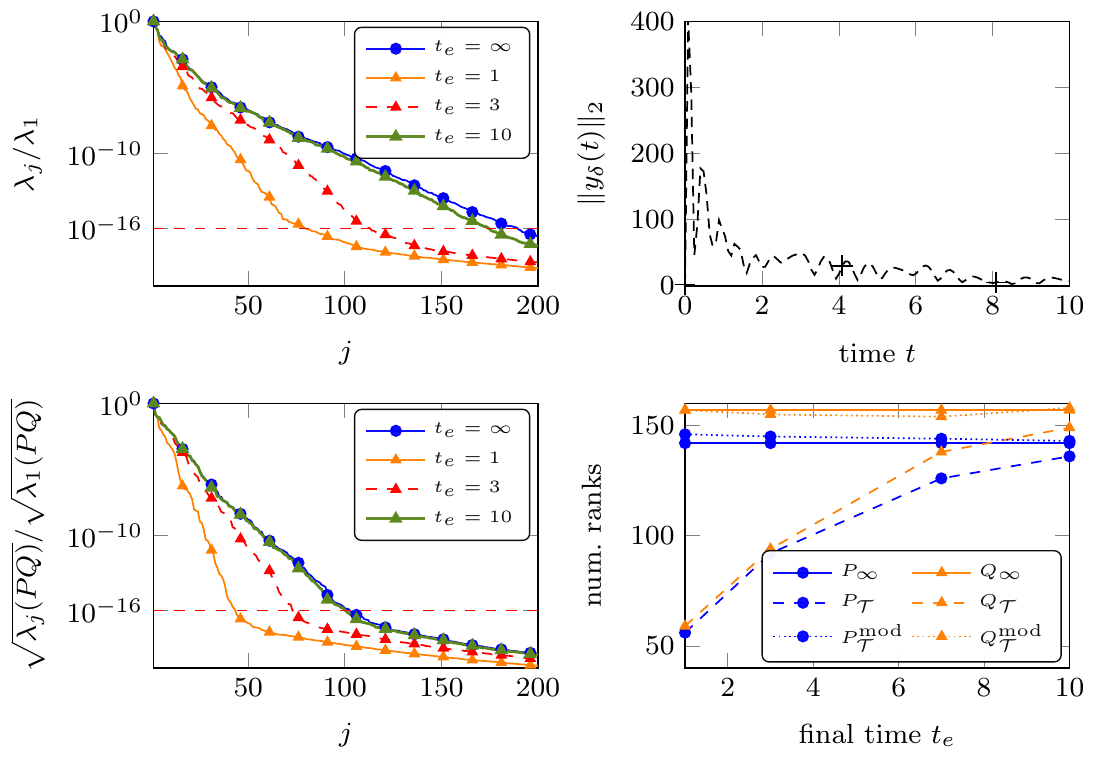}
  \caption{Scaled eigenvalues of $P_{\cT}$ (top left), Hankel and time-limited singular values $\sqrt{\lambda(P_{\cT}Q_{\cT})}$ (bottom left), 
norm of impulse response $y_{\delta}(t)$ (top right), and numerical ranks of infinite, (modified) time-limited Gramians against varying final times
$t_e$ (bottom right) for the \texttt{bips\_606} system.}\label{fig:gram_evdecay}
\end{figure}

At first we investigate how the end time $t_e$ influences the eigenvalue decay of the time-limited Gramians. The index one descriptor system \texttt{bips\_606}
is used for this experiment and reformulated into an equivalent state space system of dimension $n_f=606$ as explained in Section~\ref{ssec:gen}. This
comparatively small size allows a direct computation of the matrix exponentials and the Gramians. The top left plot in Figure~\ref{fig:gram_evdecay} shows the
scaled and ordered eigenvalues
$\lambda_j/\lambda_{1}$  of the infinite reachability Gramian $P_{\infty}$ and the time-limited one $P_{\cT}$ for $t_e=1,3,10$. Obviously, a distinctly faster
eigenvalue decay of  $P_{\cT}$ is only observed for small time values $t_e=1,3$. As the final time increases, the eigenvalues move closer to the ones
of $P_{\infty}$. The eigenvalues of the observability Gramians exhibit a similar behavior. This observation is even more drastic for the decay of the
time-limited Hankel singular values shown in the bottom left plot. For the largest value  $t_e=10$, hardly any difference to the Hankel singular values
is visible. In the top right plot the point wise norm of the impulse response $y_{\delta}(t)$ shows that after $t_e=10$, $y_{\delta}(t)$ has
already almost reached its stationary phase. This confirms the expectation that significant differences between infinite and time-limited Gramians occur
only for times $t_e$ which are
small with respect to the behavior of the impulse response. The bottom right plots shows the numerical rank of infinite, time-limited and modified time-limited
Gramians
against $t_e$. The numerical ranks of $P_{\cT},~Q_{\cT}$ clearly move towards the numerical ranks of $P_{\infty},~Q_{\infty}$ as $t_e$ increases. In contrast,
the numerical ranks of $P^{\text{mod}}_{\cT},~Q^{\text{mod}}_{\cT}$ (Section~\ref{ssec:modBT}) are always close to the ones of $P_{\infty},~Q_{\infty}$ and
appear to be largely unaffected by different values of~$t_e$. This is a very similar behavior as observed for the modified frequency-limited Gramians
in~\cite{morBenKS16}.
%%%%%%%%%%%%%%%%%%%%%%%%%%%%%%%%%%%%%%%%%%%%%%%%%%%%%%%%%%%%
\subsection{Computing low-rank factors of the time-limited Gramians}
We proceed by testing the computation of low-rank factors of the infinite and (modified) time-limited reachability Gramians by the rational Krylov subspace
method in Algorithm~\ref{alg:ksm_tlimcale}. The stopping criteria for matrix function and Gramian approximations use the thresholds $\tau_f=\tau_P=10^{-8}$. 
To save some computational cost, the projected matrix exponentials and Lyapunov equations (Steps~\ref{ksm_Bj} and~\ref{ksm_Pj} in
Algorithm~\ref{alg:ksm_tlimcale}) are only dealt with every 5th iteration step.

\begin{table}[tb]
  \footnotesize
  \centering
  \caption[Results of the numerical computation of low-rank factors of the
different Gramians.]{Results of the numerical numerical computation of low-rank factors of the
different Gramians: time horizon $t_e$, generated subspace dimension $d$, rank $r$ of the low-rank approximations, and computing time $t_{\text{rk}}$
in seconds.}
\setlength{\tabcolsep}{0.5em}
  \begin{tabularx}{1\linewidth}{|X|l|l|r|r|r|r|r|r|r|r|}
    \hline   
\multicolumn{2}{|l|}{}&\multicolumn{3}{c|}{$P_{\infty}$}&\multicolumn{3}{c|}{$P_{\cT}$
    } &\multicolumn{3}{c|}{$P_{\cT}^{\text{mod}}$}\\
    \cline{3-11}
    Example&$t_e$&$d$&$r$&$t_{\text{rk}}$&$d$&$r$&$t_{\text{rk}}$&$d$&$r$&$t_{\text{rk}}$\\
    \hline
\texttt{bips\_3078}
&3&308&245&18.33&664&131&105.89&664&241&106.17\\ 
    \hline
\texttt{vertstand}
&300&180&164&6.18&300&140&13.42&300&183&13.55\\ 
      \hline 
\multirow{2}{*}{\texttt{rail}}
&10&245&224&34.71&420&168&74.06&420&228&76.06\\  
&100&245&224&34.71&420&198&76.53&420&227&74.68\\ 
    \hline
  \end{tabularx}\label{tab:lrf_compare}
\end{table}

Table~\ref{tab:lrf_compare} summarizes the used time values $t_e$, the produced subspace dimensions $d$, the ranks $r$ of the low-rank approximations, and the
total computing times $t_{\text{rk}}$ for the Gramians $P_{\infty}$, $P_{\cT}$, and $P^{\text{mod}}_{\cT}$. Apparently, larger rational Krylov  subspaces and
approximately twice as long computation  times are needed to obtain low-rank factors of the time-limited Gramians, but the final ranks of the low-rank
approximations of the time-limited Gramians $P_{\cT}$ are in all cases smaller compared to $P_{\infty}$. The modified time-limited Gramians
$P^{\text{mod}}_{\cT}$ do not show this behavior as they have ranks similar to the infinite Gramians. For the \texttt{rail} example, increasing the time horizon
$t_e$ does not change the required subspace dimension $d$ for the approximation of the time-limited Gramian $P_{\cT}$, but its rank $r$ is clearly larger. In
all cases, no additional steps of the rational Krylov method were necessary once the approximation of $\expm{At_e}B$ was found.    
The results for the observability Gramians were largely similar. For the \texttt{bips\_3078} example, the observability Gramians
appeared to be more demanding for Algorithm~\ref{alg:ksm_tlimcale} than the reachability Gramians. 
%%%%%%%%%%%%%%%%%%%%%%%%%%%%%%%%%%%%%%%%%%%%%%%%%%%%%%%%%%%%
\subsection{Model reduction results}
Now we execute BT~\cite{morMoo81} as well as (modified) TLBT~\cite{morGawJ90,morGugA04} using the square-root method (Algorithm~\ref{alg:lrsrbt})
with the low-rank factors of the reachability and observability Gramians computed in the previous section. For this purpose, we restrict ourselves to the
reduction to fixed specified orders $r$.
The approximation quality of the constructed reduced order models is assessed via the point wise and
maximal relative error norms 
\begin{align*}
 \cE(t):=\tfrac{\|y(t)-y_r(t)\|_2}{\|y(t)\|_2},~t\leq t_f,\quad \cE_{\cT}:=\max\limits_{t\in[0,t_e]}\cE(t).
\end{align*}
of the output responses $y(t),~\ty(t)$ of original and reduced order models. We consider the response $y_{\delta}(t)$ for the 
 impulse input $u(t)=\delta(t)\mathbf{1}_m$ for all examples. Moreover, for each  used test system, the transient response with respect to an additional input
signal
$u(t)$ is also considered. We use step like input signals $u(t)=\mathbf{1}_m$ and $u(t)=50\mathbf{1}_m$ for the \texttt{bips\_3078} and, respectively,
\texttt{rail} example, and $u_*:=[5\cdot 10^{4}\cdot 0.198(\sin(t\pi/100)^2), 4, 2, 1, 3,1]^T$ for \texttt{vertstand}.

\begin{table}[tb]
  \footnotesize
  \centering
  \caption[Model reduction results by BT, TLBT, and modified TLBT.]{Model reduction results by BT, TLBT, and modified TLBT using different $t_e$ and $u(t)$:
 reduced order $r$, largest relative error $\cE_{\cT}$ in $[0,t_e]$, overall computation time $t_{\text{mor}}$, and $s\in\lbrace{0,1\rbrace}$ indicates if
reduced system is asymptotically stable or unstable.}
\setlength{\tabcolsep}{0.5em}
  \begin{tabularx}{1\linewidth}{|X|l|l|l|r|r|r|r|r|r|r|r|}
    \hline   
\multicolumn{3}{|l|}{}&\multicolumn{3}{c|}{BT}&\multicolumn{3}{c|}{TLBT}&\multicolumn{3}{c|}{mod. TLBT}\\
    \cline{4-12}
    Example&$u$&$t_e$&$\cE_{\cT}$&$t_{\text{mor}}$&$s$&$\cE_{\cT}$&$t_{\text{mor}}$&$s$&$\cE_{\cT}$&$t_{\text{mor}}$&$s$\\
    \hline
\multirow{2}{*}{
 \begin{minipage}{\linewidth}
  \texttt{bips\_3078},\\
  $r=100$
 \end{minipage}}
&$\delta$&3&5.10e-04&291.3&1&1.08e-06&331.9&0&5.15e-04&399.6&1\\ 
&1&3&6.90e-06&291.3&1&6.33e-09&331.9&0&1.20e-05&399.6&1\\  
    \hline
    \multirow{2}{*}{
 \begin{minipage}{\linewidth}
\texttt{vertstand},\\
$r=20$ 
\end{minipage}}
&$\delta$&300&8.90e-02&13.2&1&2.44e-02&30.4&1&6.48e-02&31.6&1\\ 
&$u_*$&300&8.26e-03&13.2&1&9.18e-04&30.4&1&9.95e-03&31.6&1\\ 
      \hline 
\multirow{4}{*}{
 \begin{minipage}{\linewidth}
      \texttt{rail},\\
       $r=50$ 
 \end{minipage}}      
&$\delta$&10&6.60e-01&62.2&1&5.66e-04&130.8&0&6.59e-01&137.4&1\\ 
&$\delta$&100&6.60e-01&62.2&1&7.90e-03&135.7&1&6.60e-01&136.4&1\\ 
&50&10&1.78e-03&62.2&1&6.08e-07&130.8&0&2.61e-03&137.4&1\\
&50&100&1.78e-03&62.2&1&2.57e-05&135.7&1&2.15e-03&136.4&1\\ 
    \hline
  \end{tabularx}\label{tab:reduc}
\end{table}
The results are given in Table~\ref{tab:reduc} listing the largest relative error $\cE_{\cT}$ in $[0,t_e]$ and the overall computation time $t_{\text{mor}}$,
i.e., the computation time for computing low-rank factors of both reachability and observability Gramians by Algorithm~\ref{alg:ksm_tlimcale}
plus the time for Algorithm~\ref{alg:lrsrbt} to execute the BT variants. We also indicate whether the produced reduced order models are asymptotically stable
$(s=1)$ or unstable $(s=0)$. 
For some selected settings of $u(t)$ and $t_e$, the system responses and point wise relative errors $\cE(t)$ are plotted
against
the time $t$ in
Figure~\ref{fig:response} and Figure~\ref{fig:reduct}, respectively.
Figure~\ref{fig:errvsord} shows the behavior of $\cE_{\cT}$ as the reduced order $r$
increases.

% {%
\begin{figure}[t]
  \centering
%   \tikzsetnextfilename{figure2}
% \begin{tikzpicture}
%    \input{images/bips3078_impulsres}
% \input{images/vertstand_ustar_response}
% \input{images/rail79k_impulse_resp_te10}
% \input{images/rail79k_step50_resp_te100}
% \end{tikzpicture}
\includegraphics{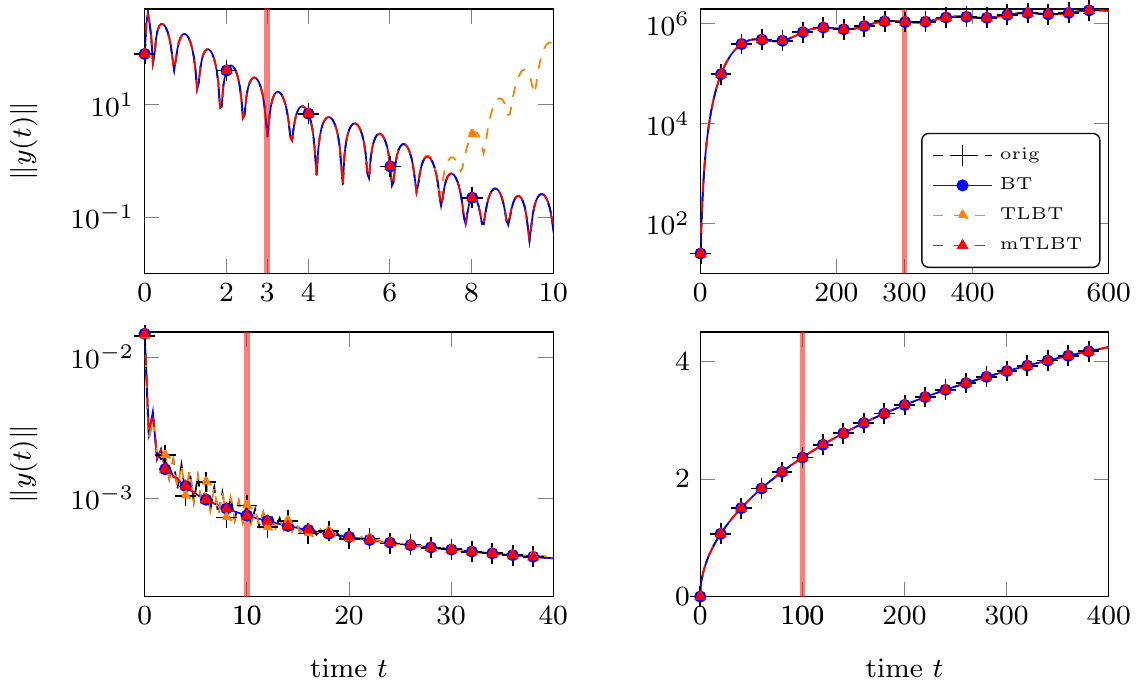}
  \caption{Responses of original and reduced systems obtained by different BT versions: (top left) \texttt{bips\_3078}, $t_e=3$, $u(t)=$impulse,
$r=100$,
(top right) \texttt{vertstand}, $t_e=300$, $u(t)=u_*$, $r=20$, (bottom left) \texttt{rail}, $t_e=10$, $u(t)=$impulse, $r=50$, (bottom left)
\texttt{rail}, $t_e=100$, $u(t)=50$, $r=50$.
}
  \label{fig:response}
\end{figure}
% }

\begin{figure}[t]
  \centering
%   \tikzsetnextfilename{figure3}
% \begin{tikzpicture}
%    \input{images/bips3078_imp_error}
% \input{images/vertstand_ustar_error_te300}
% \input{images/rail79k_impulse_error_te10}
% \input{images/rail79k_step50_error_te100}
% \end{tikzpicture}
\includegraphics{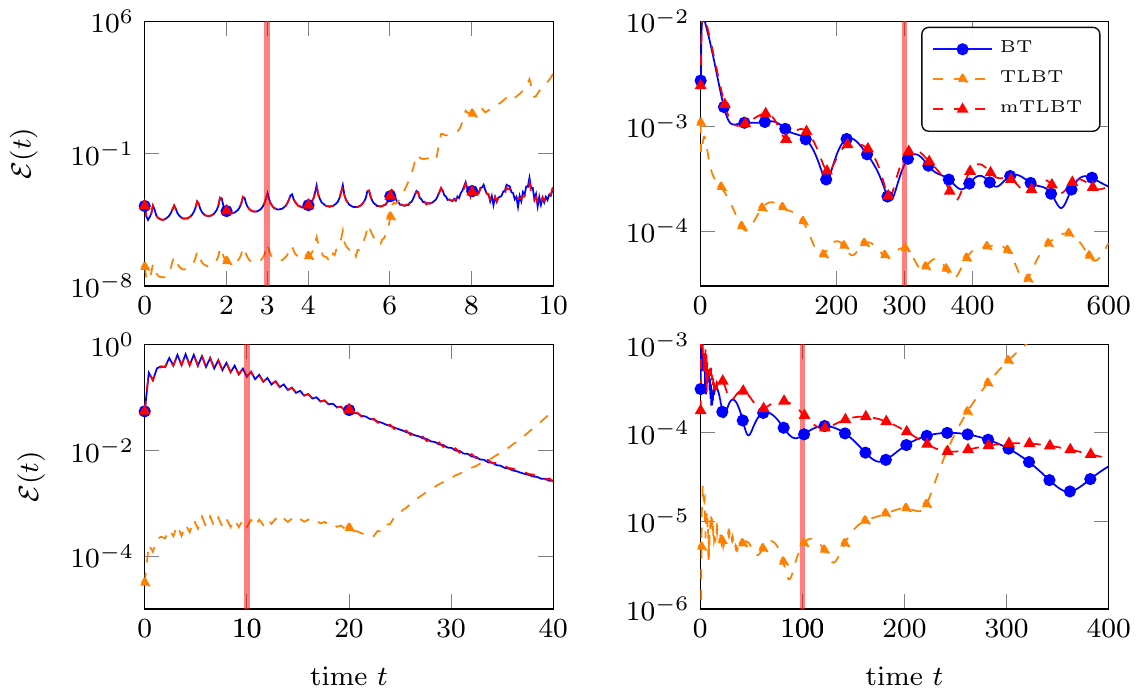}
  \caption{Relative errors $\cE(t)$ obtained by different BT versions: (top left) \texttt{bips\_3078}, $t_e=3$, $u(t)=$impulse, $r=100$,
(top right) \texttt{vertstand}, $t_e=300$, $u(t)=u_*$, $r=20$, (bottom left) \texttt{rail}, $t_e=10$, $u(t)=$impulse, $r=50$, (bottom left)
\texttt{rail}, $t_e=100$, $u(t)=50$, $r=50$.}
  \label{fig:reduct}
\end{figure}

%%%%%%%%%%%%%%%%%%%%%%%%%%%%%%%%%%%%%%%%%%%%%%%%%%%%%%%%%%%%%%%%%%%%%%%%%%%%%%
\begin{figure}[!h]
  \centering
%   \tikzsetnextfilename{figure4}
% \begin{tikzpicture}
%    \input{images/bips3078_maxerrvsord}
% \input{images/vertstand_maxerrvsord_te300}
% \input{images/rail79k_maxerrvsord_te10}
% \input{images/rail79k_maxerrvsord_te100}
% \end{tikzpicture}
\includegraphics{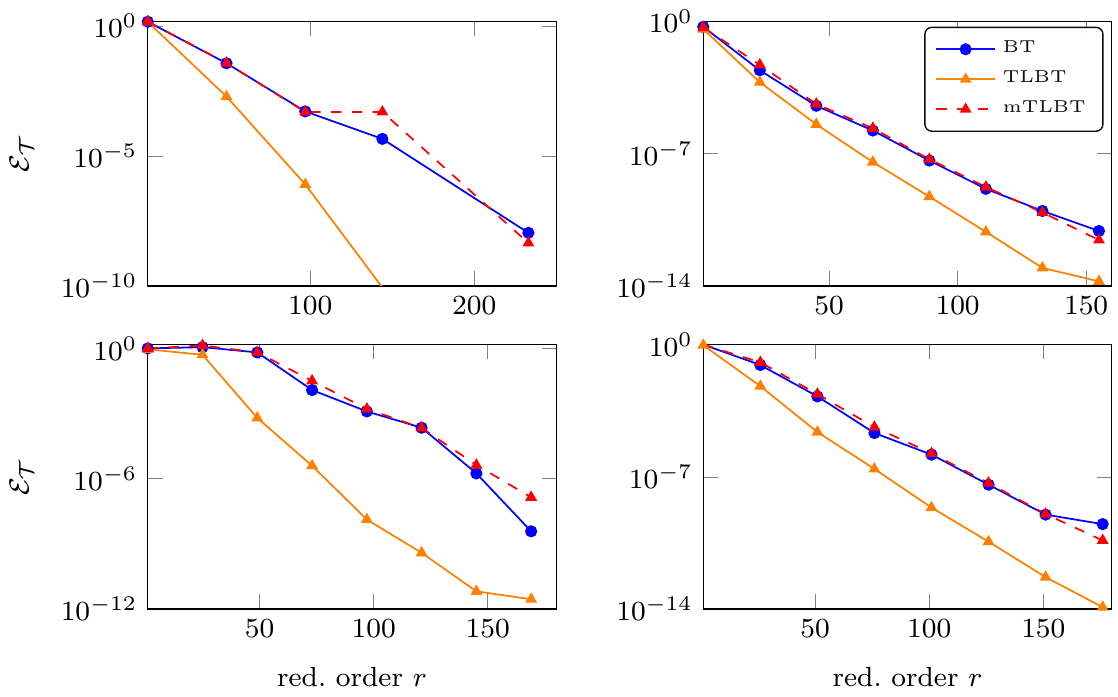}
  \caption{Maximum relative errors $\cE_{\cT}$ in $[0,t_e]$ against increasing reduced orders $r$ for different $u(t)$ and BT variants: (top left)
\texttt{bips\_3078}, $t_e=3$, $u(t)=$impulse,
(top right) \texttt{vertstand}, $t_e=300$, $u(t)=u_*$, (bottom left) \texttt{rail79k}, $t_e=10$, $u(t)=$impulse, (bottom left)
\texttt{rail79k}, $t_e=100$, $u(t)=50$.}
  \label{fig:errvsord}
\end{figure}

% The plots of $y(t)$ and $\ty(t)$ in Figure~\ref{fig:response} are visually indistinguishable in the relevant time
% interval $[0,t_e]$.  
Apparently, for the chosen orders $r$ and in the time regions of interest, the largest relative errors $\cE_{\cT}$ of the reduced order
models returned by TLBT are in most experiments more than one order of magnitude
smaller compared to standard and
modified time-limited BT. The plots in Figure~\ref{fig:errvsord} also indicate that much larger reduced order models are needed for BT and modified TLBT to
achieved the same accuracy as unmodified TLBT.  Figure~\ref{fig:reduct} shows that after leaving the time region $\cT$, TLBT delivers larger errors.
However, Table~\ref{tab:reduc} also reveals that
executing TLBT and its modified version is more time consuming
than standard BT. This is a direct consequence of the higher computation times for getting the low-rank factors of the (modified) time-limited Gramians which
was pointed out in the previous subsection (Table~\ref{tab:lrf_compare}). 

Using the concept of angles between subspaces or the modal assurance
criterion~(MAC)~\cite{Ewi00} ($MAC(x,y)=\vert y^Tx\vert^2/(\|x\|^2\|y\|^2)$) indicated that the spaces spanned by the projection matrices $T_{\text{TLBT}}$,
$S_{\text{TLBT}}$ in TLBT and $T_{\text{BT}}$, $S_{\text{BT}}$ in unrestricted BT, respectively, are different. For example, computing the MAC for the right
projection matrices $T_{\text{TLBT}}$, $T_{\text{BT}}\in\R^{n_f\times 100}$ for the \texttt{bips\_3078} system showed that only a few of the columns of  both
matrices are well correlated to each other (i.e., $MAC(T_{\text{TLBT}}e_i,T_{\text{BT}}e_j)\approx 1$ for very few $i,j\in\lbrace1,\ldots,100\rbrace$).

Moreover, albeit the higher accuracy in $[0,t_e]$, in some cases TLBT produces unstable reduced order models. 
This is especially visible in the upper left
plot of Figure~\ref{fig:response} showing the impulse responses of \texttt{bips\_3078}. After leaving $[0,t_e]$, the impulse response of reduced order model
generated by TLBT exhibits an exponential growth and departs from the original response.
As illustrated with the \texttt{rail} examples and proven in~\cite{morGawJ90}, using a higher end time $t_e$ can already cure this. Modified
TLBT does not generate unstable reduced systems, but its approximation quality in $[0,t_e]$ is very close to standard BT
without time restrictions. Taking also into account the higher computational costs of modified TLBT given in Table~\ref{tab:reduc}, the introduction of the time
restriction is 
rendered essentially redundant because no smaller errors are achieved in the targeted time region. Hence, if stability preservation in the reduced order model
is crucial, we recommend to stick to standard~BT. 

To conclude, TLBT fulfills the goal to acquire smaller errors in the desired time interval $[0,t_e]$, but at the price of somewhat larger execution times
because the computation of required low-rank Gramian factors is currently more costly. 

% \smallskip

\smallskip
Now we carry out one experiment to evaluate the approximation qualities of TLBT for nonzero $t_s$. The \texttt{vertstand} example is used and the reduced
order model should approximate the output $y(t)$ with respect to $u(t)=\delta(t)v$ and $u(t)=u_*=[5\cdot 10^4\cdot 0.198(\sin(t\pi/100)^2), 4, 2, 1, 3,1]^T$ in
the time window $\cT=[t_s,t_e]=[50,100]$.
The low-rank factors of the Gramians are computed as before using Algorithm~\ref{alg:ksm_tlimcale} with the required small extensions mentioned in
Section~\ref{ssec:multitime}.

\begin{table}[tb]
  \footnotesize
  \centering
  \caption{Results of BT, TLBT, and modified TLBT reduction to $r=15$ of \texttt{vertstand} example with respect to time frame $\cT=[t_s,t_e]=[50,100]$.}
  \begin{tabularx}{1\linewidth}{|X|r|r|r|r|r|r|r|r|r|}
    \hline   
&\multicolumn{3}{c|}{BT}&\multicolumn{3}{c|}{TLBT}&\multicolumn{3}{c|}{mod. TLBT}\\
%     \cline{3-11}
    \hline
    input $u$&$\cE_{\cT}$&$t_{\text{mor}}$&$s$&$\cE_{\cT}$&$t_{\text{mor}}$&$s$&$\cE_{\cT}$&$t_{\text{mor}}$&$s$\\
    \hline
$\delta$&7.95e-04&14.2&1&3.07e-04&37.6&0&9.21e-04&32.5&1\\ 
$u_*$&6.34e-03&14.2&1&1.78e-02&37.6&0&6.85e-03&32.5&1\\ 
\hline
\end{tabularx}\label{tab:reduc_ts}
\end{table}

\begin{figure}[t]
  \centering
%   \tikzsetnextfilename{figure5}
%   \begin{tikzpicture}
%      \input{images/vertstand_impuls_error_ts50te100}
% \input{images/vertstand_ustar_error_ts50te100}
% \input{images/vertstand_maxerrvsord_ts50te100}
%   \end{tikzpicture}
  \includegraphics{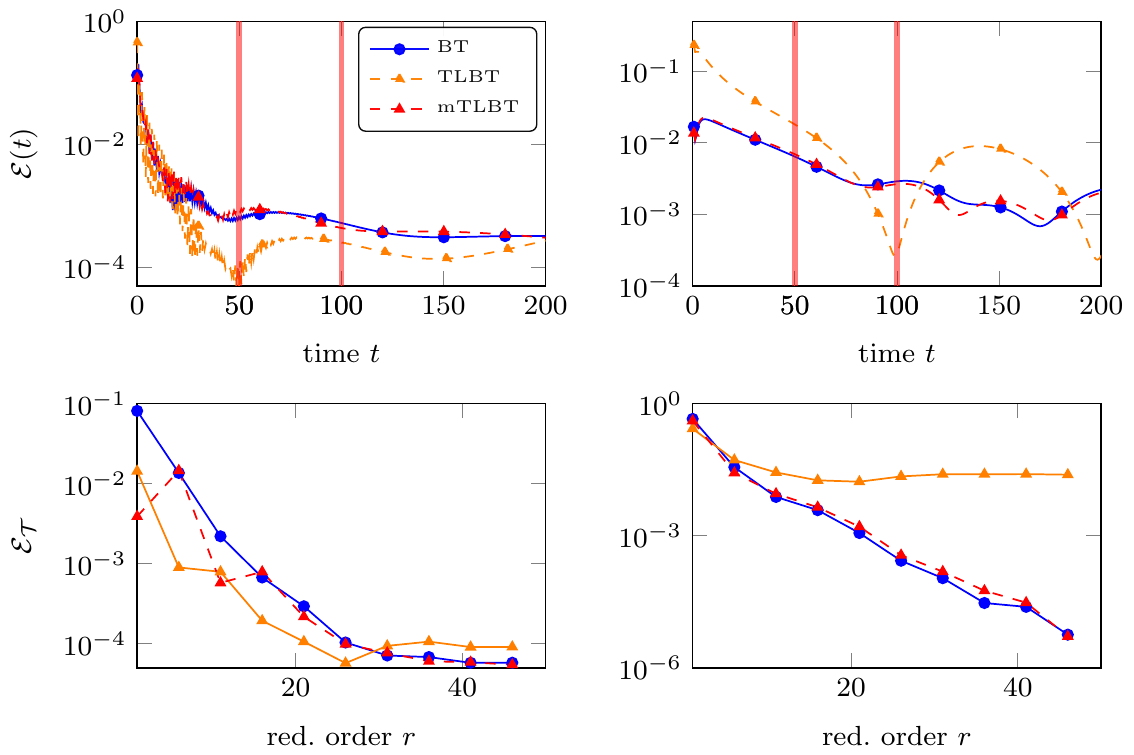}
  \caption{Results for the reduction of the \texttt{vertstand} example in $[t_s,t_e]=[50,100]$. The top plots show the relative error norms against $t$ for
$u(t)=\delta(t)v$ (left) and $u(t)=u_*$ (right) after a reduction to $r=15$. The behavior of $\cE_{\cT}$ in $[t_s,t_e]$ for increasing reduced dimensions $r$
is illustrated in the bottom plots.}
  \label{fig:results_tste}
\end{figure}
The results are summarized in Table~\ref{tab:reduc_ts} and Figure~\ref{fig:results_tste} illustrates the relative errors plots as well as the
largest relative error in $\cT$ against the reduced order $r$. Compared to standard BT, TLBT delivers, as expected, more accurate results of the impulse
response, but fails for $u(t)=u_*$. Moreover, the bottom left plot Figure~\ref{fig:results_tste} shows that the decay of
$\cE_{\cT}$ with respect to
the impulse response is less monotonic as in the case $t_s=0$ s.t. for some reduced orders $r$, especially larger ones, standard BT outperforms TLBT.  
The failure of TLBT to approximate the transient response for arbitrary inputs $u(t)$ was also observed in other experiments with different time intervals and
examples, and occasionally even for the impulse response. In experiments with smaller systems, using exact matrix exponentials and Gramian factors did
not yield any improvement of the reduction results such that the problems with TLBT are most likely not a result of inaccurate Gramian approximations.
In summary, although TLBT in the current form does indeed deliver significantly more accurate reduced order models in the time intervals $[0,t_e]$, the  same
cannot be said when time intervals $[t_s,t_e]$, $t_s>0$ are considered. Improving the performance of TLBT in this scenario is therefore subject of further
research.

% \smallskip
%
As final experiment we briefly test the application of TLBT for the reduction of unstable systems. For this purpose a modification of the \texttt{bips\_606}
example is used with $\hat A:=A+0.2 M$ leading to $\max\Real{\lambda}=0.323$, $\lambda\in\Lambda(\hat A)$. The final time is set to $t_e=5$ and
a reduced order model with $r=100$ is generated. Due to the comparatively small system dimension ($n_f=606$),
the matrix exponentials and Gramians could be computed by direct, dense methods for these tasks. Moreover, the employed rational Krylov methods from the
experiments before were not able to compute the required low-rank Gramians factors. The impulse response of exact and reduced system and the relative error are
illustrated in Figure~\ref{fig:results_bips_us}.
\begin{figure}[t]
  \centering
%     \tikzsetnextfilename{figure6}
%       \begin{tikzpicture}
%     \input{images/bips606_unstable_impresp}
% \input{images/bips606_unstable_imp_error}    
%       \end{tikzpicture}
      \includegraphics{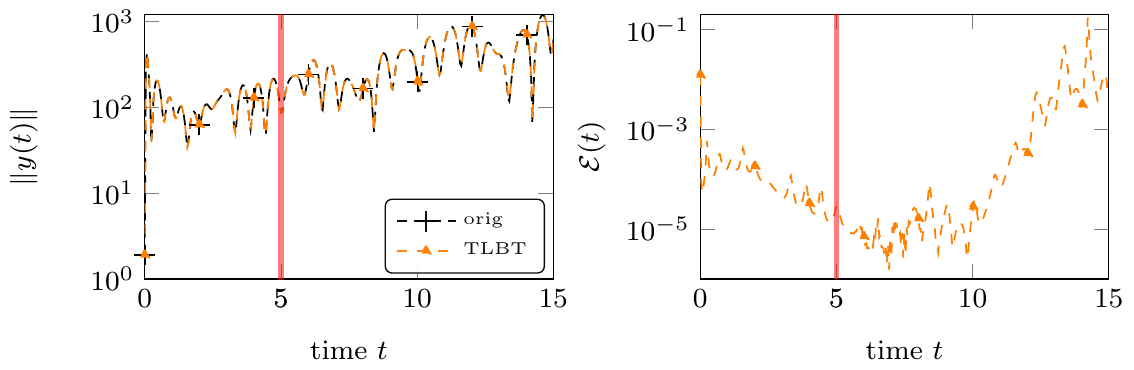}
  \caption{Results for the reduction of the unstable variation of the \texttt{bips\_606} example for $t_e=5$, $r=100$, $u=\delta(t)v$ (left: system
response,
right: relative error).}
  \label{fig:results_bips_us}
\end{figure}
Apparently, TLBT was able to reduce this mildly unstable system to a reduced order model with maximal relative error $\cE_{\cT}\approx 1.2\cdot10^{-2}$ in 
$[0,t_e]$. Hence, provided $t_e$ is chosen in a reasonable way, e.g., with a sufficient distance to the exponential growth of $y(t)$, TLBT appears
to be a potential candidate from model order reduction of unstable systems. However, the applications to large-scale unstable systems is currently still
difficult because algorithms for computing low-rank solutions of Lyapunov equations usually require that $A$ is (anti)stable. Advances in this direction are,
therefore, necessary to pursue this type of reduction further.
%%%%%%%%%%%%%%%%%%%%%%%%%%%%%%%%%%%%%%%%%%%%%%%%%%%%%%%%%%%%%%%%%%%%%%%%%%%%%%%%%%%%%%%%%%%%%%%
\section{Conclusions}\label{sec:concl}
BT model order reduction for large-scale systems restricted to finite time intervals~\cite{morGawJ90} was investigated. 
The resulting Lyapunov equations that have to be solved numerically also include the matrix exponential in their inhomogeneities.
We first showed that the difference of time-limited and infinite Gramians decays for increasing times in a similar way as the impulse
response of the underlying system. Hence, for small time intervals, a reduced numerical rank of the time-limited  Gramians can be observed.
Future research should further investigate the influence of the chosen end time on the eigenvalue decay of the Gramians.
  
As in frequency-limited BT~\cite{morBenKS16}, we proposed to handle the matrix exponentials and the Lyapunov equations by an efficient
rational Krylov subspace method incorporating a subspace recycling idea. While this numerical approach already led to satisfactory results, there is some
room for improvement, especially regarding the approximation of the action of the matrix exponential. In this context, further work could include, for instance,
enhanced strategies like tangential directions~\cite{DruSZ14}, different inner products~\cite{FroLS17}, or using different
methods~\cite{AlmH11,CalKOetal14} altogether.  

The numerical reduction experiments indicated that TLBT is able to acquire several orders of magnitude more accurate reduced order models in time intervals
of the form $[0,t_e]$ at a somewhat higher, although still comparable, numerical effort. 
Similar techniques were applied for stability preserving modified TLBT~\cite{morGugA04} which, however, could not keep up with standard or time-limited
BT in terms of efficiency and accuracy of the reduced order models. Hence, we recommend to use standard BT if the preservation of stability is an irrevocable
goal. For keeping the high accuracy in the time-interval of interest, a different way to make TLBT a stability preserving method has to be found. 
TLBT for time regions  $[t_s,t_e]$ with nonzero start times $t_s>0$ provided much worse results than with $t_s=0$, except for approximating the impulse
response. In the form introduced in~\cite{morGawJ90}, TLBT appears to be incapable of producing good reduced order models with respect to $[t_s,t_e]$ and, thus,
further
research is necessary in this direction. The results in~\cite{morBeaGM17,morHeiRA11} might be one possible ingredient for this.
A short experiment also indicated that TLBT can be employed to reduce unstable systems. The efficient computation of low-rank Gramian factors in this
case is
currently not as advanced as in the stable situation, making this a further interesting research topic.

\textbf{Acknowledgements}
 I thank the referees for their helpful comments. Moreover, I am grateful for the constructive discussions with Maria Cruz Varona, Serkan Gugercin,  and 
Stefan Guettel.

% \bibliographystyle{siam}   
% \bibliographystyle{spmpsci}
% \bibliography{csc,mor,tlbt}
\end{document}